\documentclass[12pt]{amsart}
\usepackage[OT4]{fontenc}
\usepackage{graphics,graphicx}
\usepackage{amsmath}
\usepackage{amssymb}
\usepackage {amsthm}

\usepackage{xcolor}

\usepackage{enumerate}
\usepackage{hyperref}
\usepackage[cp1251]{inputenc}
\usepackage[all]{xy}
\usepackage{cleveref}
\usepackage{lipsum}

\usepackage[active]{srcltx}

\newtheorem{theorem}{Theorem}

\newtheorem{lemma}{Lemma}

\newtheorem*{proposition*}{Proposition}
\newtheorem{definition}{Definition}

\newtheorem{remark}{Remark}

\makeatletter
\renewcommand{\subsection}{\@startsection{subsection}{2}{0mm}{-\baselineskip}{-5pt}{\it \bf}}
\makeatother

\title{Morita equivalent unital locally matrix algebras}
\author{Oksana Bezushchak and Bogdana Oliynyk }

\thanks{The second author was partially supported by the grant for scientific researchers of the ``Povir u sebe'' Ukranian Foundation.}

\begin{document}
	
	\maketitle
	
	\address{Faculty of Mechanics and Mathematics,
	Taras Shevchenko National University of Kyiv,
	Volodymyrska, 60, Kyiv 01033, Ukraine \\ Department of Mathematics, National University
		of Kyiv-Mohyla Academy, Skovorody St. 2, Kyiv,
		04070, Ukraine                     }
	
	\email{bezusch@univ.kiev.ua, oliynyk@ukma.edu.ua}

	\keywords{Keyword: locally matrix algebra, Steinitz number,  Morita equivalence}               
	
	\subjclass{2010}{Mathematics Subject Classification: 03C05, 03C60}
	

	
	\begin{abstract}
		We describe Morita equivalence of unital locally matrix algebras in terms of their Steinitz parametrization. Two countable dimensional unital locally matrix algebras are Morita equivalent if and only if their Steinitz numbers are rationally connected. For an arbitrary uncountable dimension $\alpha$ and an arbitrary not locally finite Steinitz number $s$ there exist unital locally matrix algebras $A$, $B$ such that $\dim_{F}A=\dim_{F}B=\alpha$,  $\mathbf{st}(A)=\mathbf{st}(B)=s$, however, the algebras  $A$, $B$ are not Morita equivalent. 
		
	\end{abstract}

\section{Introduction}

 Let $F$ be a ground  field. Throughout the  paper   we consider unital associative $F$--algebras. An algebra $A$ with a unit $1_A$ is called a {\it unital locally matrix algebra} if an arbitrary finite collection of elements $a_1,$ $\ldots,$ $a_s \in A$  lies in a subalgebra $B$,  $1_A\in B \subset A$,  that is  isomorphic to a matrix algebra $M_n(F),$ $n\geq 1.$

The idea of parametrization of unital locally matrix algebras   with Steinitz numbers was introduced by J.~G.~Glimm \cite{Glimm}. Diagonal locally simple Lie algebras of countable dimension were parametrized with Steinitz numbers by  A.~A.~Baranov and A.~G.~Zhilinskii  in \cite{Baranov2}, \cite{Baranov1}.
The extension of these results to regular relation structures was done   in \cite{Sushch2}.

In this paper we apply Steinitz parametrisation  to Morita equivalence classes of unital locally matrix algebras.
We show that two countable dimensional unital locally matrix algebras are Morita equivalent if and only if  their Steinitz numbers  are rationally connected. This result does not extend to the uncountable case. Moreover, for an arbitrary uncountable dimension $\alpha$ and an arbitrary not locally finite  Steinitz number $s$ there exist  unital locally matrix algebras $A$, $B$ such that $\dim_{F}A=\dim_{F}B=\alpha$,  $\mathbf{st}(A)=\mathbf{st}(B)=s$, however, the algebras  $A$, $B$ are not Morita equivalent.

\section{Preliminaries}


Let $ \mathbb{P} $ be the set of all primes and $ \mathbb{N} $ be the set of all positive integers. A
  {\it Steinitz } or ``supernatural'' number (see \cite{ST})  is an infinite formal
product of the form
\begin{equation}\label{1}
 \prod_{p\in \mathbb{P}} p^{r_p} \ , \end{equation}
where   $ r_p\in  \mathbb{N} \cup \{0,\infty\}$ for all $p\in \mathbb{P}.$
The product of two Steinitz numbers
$$ \prod_{p\in \mathbb{P}} p^{r_p} \ \text{ and } \  \prod_{p\in \mathbb{P}} p^{k_p}  $$
 is a Steinitz number
$$ \prod_{p\in \mathbb{P}} p^{r_p+k_p}  \ ,  $$
where we assume, that  $t+\infty=\infty+t=\infty+\infty=\infty$ for all non negative integers $t$.

Denote by $ \mathbb{SN} $ the set of all Steinitz numbers. Note, that the set of all positive integers $ \mathbb{N}$ is a subset of $\mathbb{SN}$.

 A  Steinitz number \eqref{1} is called {\it locally finite} if  $r_p \neq \infty $ for any $p \in  \mathbb{P} $. The numbers  $ \mathbb{SN}\setminus   \mathbb{N}  $  are called {\it infinite}  Steinitz numbers.

J.~G.~Glimm \cite{Glimm} parametrised  countable dimensional locally matrix algebras   with Steinitz numbers. In \cite{BezOl}, \cite{BezOl_2} we  studied Steinitz numbers of unital locally matrix algebras of arbitrary dimensions.

Let $A$ be an infinite dimensional  locally matrix algebra with a unit $1_A$ over a field $F$ and let $D(A)$ be  the  set of all positive integers $n$ such that  there is a subalgebra  $A'$, $1_A \in A'\subseteq A$, $A' \cong  M_n(F)$.

\begin{definition}
\label{St_number}
The  least common multiple of the set $D(A)$ is called the Steinitz  number $\mathbf{st}(A)$  of the algebra $A$.
\end{definition}

Given two unital locally matrix algebras $A$ and $B$  their tensor product $A \otimes_F B$ is a unital  locally matrix algebra and
$ \mathbf{st}(A \otimes_F B)=\mathbf{st}(A) \cdot \mathbf{st}(B)$ (see  \cite{BezOl}, \cite{BezOl_2}). In particular, a matrix algebra $M_k(A)$ is a unital locally matrix algebra and $\mathbf{st}(M_k(A))= k \cdot \mathbf{st}(A).$

\begin{theorem}[\cite{Glimm}, see also \cite{Sushch2}]
	\label{teorBOS}
If $A$ and $B$  are unital locally matrix algebras of countable dimension then $A$ and $B$ are isomorphic if and only if  $\mathbf{st}(A)=\mathbf{st}(B)$.
\end{theorem}

Let $A$ be an  algebraic system. The universal elementary theory $UTh(A)$ consists of universal closed  formulas (see \cite{Malcev}) that are valid on $A$. The systems $A$ and $B$ of the same signature are universally equivalent if $UTh(A)=UTh(B)$.

In \cite{BezOl} we showed that for unital locally matrix algebras $A,$ $B$ of dimension $> \aleph_0$  the equality  $\mathbf{st}(A)=\mathbf{st}(B)$ does not necessarily  imply that $A$ and $B$ are isomorphic. However, $\mathbf{st}(A)=\mathbf{st}(B)$ is equivalent to  $A,$ $B$ being universally equivalent.

\section{Morita equivalence}

\begin{definition}
\label{Morita}
Two unital algebras $A,$ $B$ are called Morita equivalent if categories of their left modules are equivalent.
\end{definition}

Let $e\in A$ be an idempotent. We refer to the subalgebra $e A e$  as a {\it corner} of the algebra $A.$ An idempotent $e\in A$ is said to be {\it full} if $AeA=A.$ K.Morita \cite{Morita} (see also \cite{Lam}) proved that the algebras $A,$ $B$ are Morita equivalent if and only if there exists $n \geq 1$  and a full idempotent $e$ in the matrix algebra $M_n (A) $ such that $B \cong e M_n (A) e.$ Thus $B$ is isomorphic to a corner of the algebra $M_n(A).$

We say that a property $P$ is \emph{Morita invariant} if any two Morita equivalent algebras do satisfy or do not satisfy $P$ simultaneously.

An algebra $A$ is a tensor product of finite dimensional matrix algebras if $$A\cong \otimes_{i\in I} A_i, \ \ A_i\cong M_{n_i}(F), \ n_i \geq 1.$$  Every tensor product of finite dimensional matrix algebras is a locally matrix algebra.   G.~K\"{o}the  \cite{Koethe} showed that the reverse is true for countable dimensional algebras. A.G.Kurosh \cite{Kurosh} (see also \cite{BezOl_2}, \cite{Kurochkin}) constructed examples of locally matrix algebras that do not decompose into a tensor product of finite dimensional matrix algebras.

\begin{lemma} 	\label{1L}
\begin{enumerate}
  \item[$(1)$] Being a locally matrix algebra is a Morita invariant property.
  \item[$(2)$] Being a  tensor product of finite dimensional matrix algebras is a Morita invariant property.
\end{enumerate}
\end{lemma}
\begin{proof} Let algebras $A,$ $B$ be Morita equivalent. Then there exists $n \geq 1$ and a full idempotent $e\in M_n (A)$ such that $B\cong e M_n (A) e.$  If the algebra $A$ is locally matrix then so is the matrix algebra $M_n (A).$  J.Dixmier \cite{Dixm} showed that a corner of a locally matrix algebra is a locally matrix algebra. Hence $B$ is a locally matrix algebra.

Now suppose that $A\cong \otimes_{i\in I} A_i,$ $ A_i\cong M_{n_i}(F),$ $n_i \geq 1.$  Then  $$M_n(A)\cong M_n (F)\otimes_F A \cong M_n (F)\otimes (\otimes_{i\in I} A_i ). $$  There exists a finite subset $I_0 \subset I,$ $\mid I_0 \mid < \infty,$ such that $e \in M_n (F)\otimes (\otimes_{i\in I_0} A_i ). $  As above, the corner $e ( M_n (F)\otimes (\otimes_{i\in I_0} A_i )) e $ is a  matrix algebra. Hence, $$B \cong e M_n(A) e \cong e ( M_n (F)\otimes (\otimes_{i\in I_0} A_i )) e \otimes  (\otimes_{i\in I\setminus I_0} A_i ) , $$  which  completes the proof of the Lemma.
\end{proof}

\begin{definition}
	\label{rat_connec}
We say that nonzero  Steinitz numbers $s_1,$ $s_2$ are rationally connected if there  exists a rational number $q\in \mathbb{Q}$ such that  $s_2= q \cdot s_1 .$
\end{definition}

\begin{theorem} \label{teor_1}
\begin{enumerate}
  \item[$1)$] If  unital locally matrix algebras $A,$ $B$ are Morita equivalent then their Steinitz numbers  $ \mathbf{st}(A),$ $\mathbf{st}(B)$ are rationally connected.
  \item[$2)$] If the locally matrix algebras $A,$ $B$ are  countable dimensional then they are  Morita equivalent if and only if  $\mathbf{st}(A),$ $\mathbf{st}(B)$  are rationally connected.
  \item[$3)$] For an arbitrary not locally finite Steinitz number $s$ there exist not Morita equivalent locally matrix algebras $A,$ $B$ of arbitrary uncountable dimensions such that  $ \mathbf{st}(A)=\mathbf{st}(B).$
  \item[$4)$] For a countable dimensional locally matrix algebra $A$ the Morita equivalence class of $A$ is  countable up to isomorphism. For a  locally matrix algebra of arbitrary dimension the Morita equivalence class is countable up to universal equivalence.
\end{enumerate}
\end{theorem}

\begin{remark} Countability of Morita equivalence classes of finitely presented algebras was discussed in \emph{\cite{Zel_Alanm}},  \emph{\cite{Berest}},  \emph{\cite{Futorny}}.\end{remark}

Let $A$ be a  locally matrix algebra, let $a\in A.$ There exists a subalgebra $1\in A_1 < A, $ $a\in A_1,$  such that $A_1 \cong M_n(F),$ $n\geq 1.$  Let $r$ be the rang of the matrix $a$ in $A_1.$ Let $$r(a)= \frac{r}{n}, \ \ 0\leq r(a) \leq 1.$$ V.M.Kurochkin \cite{Kurochkin}  noticed that the number $r(a)$  does not depend on a choice of the subalgebra $A_1.$  We will call $r(a)$ the {\it relative rang} of the element $a.$

\begin{lemma} 	\label{2L} Let $e$ be an idempotent of a  locally matrix algebra $A.$ Then $ \mathbf{st}(e A e)=r(e)\cdot \mathbf{st}(A).$
\end{lemma}
\begin{proof}  Consider the family of all  matrix subalgebras $1\in A_i < A,$ $A_i \cong M_{n_i}(F),$ $i\in I,$ such that $e\in A_i.$ Then $\mathbf{st}(A)= \text{lcm} (n_i, i\in I).$ The rang of the matrix $e$ in $A_i$ is equal to $r(e) \cdot n_i.$ Hence $$e A_i e\cong M_{r(e)\cdot n_i} (F) \ \ \text{ and } \ \ \mathbf{st}(e A e)= \text{lcm} (r(e)\cdot n_i, i\in I) =r(e)\cdot \mathbf{st}(A).$$
\end{proof}

\begin{proof}[Proof of Theorem \ref{teor_2}]\label{teor_2}

1) Let $A,$ $B$ be  locally matrix algebras that are Morita equivalent. Hence \cite{Lam} there exists $k\geq 1$ and an idempotent $e \in M_k (A) $  such that $B \cong e M_k (A) e.$  Let $r(e) $ be the relative rang of the idempotent $e$ in the locally matrix algebra $M_k (A).$ By Lemma \ref{2L}   $$ \mathbf{st}(B)= r(e)\cdot  \mathbf{st}(M_k (A))= r(e)\cdot k \cdot \mathbf{st}(A).$$ Since the number $r(e)\cdot k$ is rational it follows that the  Steinitz numbers $ \mathbf{st}(A),$ $\mathbf{st}(B)$ are rationally connected.

2) Let $A,$ $B$ be countable dimensional locally matrix algebras. Suppose that their Steinitz numbers $ \mathbf{st}(A),$ $\mathbf{st}(B)$ are rationally connected. Our aim is to prove that the algebras $A,$ $B$ are Morita equivalent. There exist integers $k,$ $l\geq 1$ such that  $ k\cdot \mathbf{st}(A)=l \cdot \mathbf{st}(B).$ Consider the  matrix algebras $M_k(A)$ and $M_l(B).$ We have $$\mathbf{st}(M_k(A))= k \cdot \mathbf{st}(A)=l \cdot \mathbf{st}(B)=\mathbf{st}(M_l(B)).$$  By Glimm's Theorem \cite{Glimm}  the algebras $M_k (A)$  and   $M_l (B)$ are isomorphic. Hence the algebras $A,$ $B$ are   Morita equivalent.

3) Let $S$ be a not locally finite Steinitz number. In \cite{BezOl_2} (see also \cite{BezOl} and \cite{Kurosh}) we showed that there exists a  locally matrix algebra $A$ of an arbitrary  uncountable dimension $\alpha$  such that $\mathbf{st}(A)=s$ and $A$ is not isomorphic to a tensor product of finite  dimensional matrix algebras. It is easy to see that there exists a  locally matrix algebra $B$ of dimension $\alpha$ such that $\mathbf{st}(B)=s$  and $B$ is  isomorphic to a tensor product of finite dimensional matrix algebras.  By Lemma \ref{1L} (2) the algebras $A,$ $B$ are not Morita equivalent.

4) For a countable dimensional locally simple algebra  $A$ all algebras in its Morita equivalence class have Steinitz numbers $q \cdot \mathbf{st}(A),$ where $0\neq q$ is a rational number, and are uniquely determined by their   Steinitz numbers. This implies that the Morita equivalence class of $A$ is  countable.

If the algebra $A$ is not necessarily countable dimensional then Steinitz numbers $q \cdot \mathbf{st}(A) $  determine  universal elementary theories of  algebras  in this class (see \cite{BezOl}). Hence the  Morita equivalence class of $A$ is countable up to elementary equivalence.  This completes  the proof of Theorem \ref{teor_2}.
\end{proof}

If nonzero Steinitz numbers $s_1,$ $s_2$ are rationally connected then it makes
sense to talk about their ratio $q=\frac{s_2}{s_1}$ which is a rational number.

For a countable dimensional locally matrix algebra $A$ its Morita equivalence class
{\it is ordered}: for algebras $A_1,$ $A_2$ in this class we say that $A_1 < A_2$ if
$$\frac{\mathbf{st}(A_1)}{\mathbf{st}(A_2)}<1.$$

\begin{lemma} Let $A_1,$ $A_2$ be countable dimensional Morita equivalent locally
matrix algebras. Then $$\frac{\mathbf{st}(A_1)}{\mathbf{st}(A_2)}<1 \text{ if and only if } A_1 \text{ is isomorphic
to a proper corner of } A_2.$$
\end{lemma}
\begin{proof} If $A_1\cong e A_2 e,$ where $e$ is a proper idempotent of the algebra
$A_2$ then $\mathbf{st}(A_1)=r(e) \mathbf{st}(A_2)$ by Lemma \ref{2L}. Hence $$\frac{\mathbf{st}(A_1)}{\mathbf{st}(A_2)}=r(e)<1.$$

Now let $$\frac{\mathbf{st}(A_1)}{\mathbf{st}(A_2)}= \frac{m}{n}<1,$$ where $m,$ $ n$ are relatively
prime integers. Then $n$ is a divisor of $\mathbf{st}(A_2).$ Hence the algebra $A_2$ contains
a subalgebra $1\in A_2^{'}<A_2,$ $A_2^{'}\cong M_n(F).$ Hence (see \cite{Kurosh}) $$A_2\cong A_2^{'} \otimes_F C \cong M_n (C),$$  where $C$ is the centralizer of the
subalgebra $A_2^{'}$ in $A_2.$ Consider the idempotent $e= diag(\underbrace{1,1,...,1}_m,0,...,0) \in M_n (C).$  By Lemma \ref{2L} $$ \mathbf{st}(e M_n (C) e)=\frac{m}{n} \ \mathbf{st}(A_2)= \mathbf{st}(A_1).$$  By
Glimm’s Theorem $A_1$ is isomorphic to a corner of $M_n (C),$ hence to a corner of $A_2.$
\end{proof}

\end{document}